%% file: main.tex
\documentclass[a4paper,10pt]{article}
\usepackage[T2A]{fontenc}
\usepackage[utf8]{inputenc}

\usepackage[margin=2.5cm]{geometry}

\usepackage{amsmath, amssymb, amsthm}
\usepackage{hyperref}
\usepackage{xcolor}
\usepackage{color, colortbl}
\usepackage{tikz}
\usepackage{graphicx}

\usepackage[h]{esvect}

\usepackage[shortlabels]{enumitem}

\usepackage{chngcntr}

\makeatletter
\newcommand{\customitem}[1]{%
\item[#1]\protected@edef\@currentlabel{#1}%
}
\makeatother

\usepackage{hhline}
\usepackage{diagbox}
\usepackage{multirow}

\definecolor{myred}{rgb}{0.81, 0.06, 0.13}
\definecolor{myblue}{rgb}{0,0.5,0.6}

\def\sHC{s_k^{\mathrm{HC}}}

\def\SHC{S_k^{\mathrm{HC}}}
\def\SPM{S_k^{\mathrm{PM}}}
\def\SST{S_k^{\mathrm{ST}}}
\def\SHP{S_k^{\mathrm{HP}}}
\def\x{\mathbf{x}}
\def\y{\mathbf{y}}
\def\z{\mathbf{z}}
\def\d{\mathbf{d}}
\def\a{\mathbf{a}}
\def\b{\mathbf{b}}
\def\c{\mathbf{c}}
\def\p{\mathbf{p}}
\def\q{\mathbf{q}}
\def\u{\mathbf{u}}
\def\vv{\mathbf{v}}
\def\m{\mathbf{m}}
\def\fone{\mathbf{1}}
\def\fzero{\mathbf{0}}

\def\R{{\mathbb R}}

\DeclareMathOperator{\diam}{diam}
\DeclareMathOperator{\hB}{hB}
\DeclareMathOperator{\B}{B}

\theoremstyle{plain}
\newtheorem{theorem}{Theorem}[section]
\newtheorem{corollary}[theorem]{Corollary}
\newtheorem{proposition}[theorem]{Proposition}
\newtheorem{lemma}[theorem]{Lemma}

\newtheorem{conjecture}[theorem]{Conjecture}

\theoremstyle{definition}
\newtheorem{definition}[theorem]{Definition}
\newtheorem{example}[theorem]{Example}

\theoremstyle{remark}
\newtheorem{remark}[theorem]{Remark}

\begin{document}

\title{Bollob\'as--Meir TSP Conjecture Holds Asymptotically}
\author{Alexey Gordeev}
\date{}

\maketitle

\begin{abstract}
In 1992, Bollob{\'a}s and Meir showed that for every $k \geq 1$ there exists a constant $c_k$ such that, for any $n$ points in the $k$-dimensional unit cube $[0, 1]^k$, one can find a tour $x_1, \dots, x_n$ through these $n$ points with $\sum_{i = 1}^n |x_i - x_{i + 1}|^k \leq c_k$, where $x_{n + 1} = x_1$ and $|x - y|$ is the Euclidean distance between $x$ and $y$.
Remarkably, this bound does not depend on $n$, the number of points.
They conjectured that the optimal constant is $c_k = 2 \cdot k^{k / 2}$ and showed that it cannot be taken lower than that.
This conjecture was recently revised for $k = 3$ by Balogh, Clemen and Dumitrescu, who showed that $c_3 \geq 2^{7/2} > 2 \cdot 3^{3/2}$.
It remains open for all $k > 2$, with the best known upper bound $c_k \leq 2.65^k \cdot k^{k / 2} \cdot (1 + o_k(1))$.

We significantly narrow the gap between lower and upper bounds on $c_k$, reducing it from exponential to linear.
Specifically, we prove that $c_k \leq  2\mathrm{e}(k + 1) \cdot k^{k / 2}$ and $c_k = k^{k / 2} \cdot (2 + o_k(1))$, the latter establishing the conjecture asymptotically.
We also obtain analogous results for related problems on Hamiltonian paths, spanning trees and perfect matchings in the unit cube.
Our main tool is a new generalization of the ball packing argument used in earlier works.
\end{abstract}

\section{Introduction}
\label{sec:intro}

Throughout the text, $k$ is a positive integer.
For points $\a = (a_1, \dots, a_k), \b = (b_1, \dots, b_k) \in \R^k$ denote by $\langle \a, \b \rangle := \sum_{i = 1}^k a_i b_i$ their dot product, and by $|\a| := \sqrt{\langle \a, \a \rangle}$ the Euclidean length of $\a$.
For a finite set $X \subseteq \R^k$, let $K(X)$ be the complete graph on the vertex set $X$.
We will say that $E$ is a set of edges on $X$ if $E$ is a subset of the set of edges of $K(X)$.
Similarly, we will say that $E$ is a \textit{Hamiltonian cycle on $X$} if it is a Hamiltonian cycle in $K(X)$.
For a set $E$ of edges on $X$ and an edge $e = \a\b \in E$, denote by $|e| := |\a - \b|$ the Euclidean distance between endpoints of $e$, and for a positive integer $t$ define $S_t(E) := \sum_{e \in E} |e|^t$.
The following theorem is due to Newman~\cite[Problem 57]{newmanProblemSeminar1982}.

\begin{theorem}[Newman]\label{thm:Newman}
For any finite set $X \subseteq [0, 1]^2$ of points in the unit square,
there exists a Hamiltonian cycle $H$ on $X$ with
\[
S_2(H) = \sum_{e \in H} |e|^2 \leq 4.
\]
\end{theorem}

Notably, this bound does not depend on the size of the set $X$.
Point sets in Figure~\ref{fig:exNewman} show that the constant 4 is best possible.
In 1987, Meir~\cite[Remark~(ii)]{meirGeometricProblemInvolving1987} asked for a higher-dimensional analogue of Theorem~\ref{thm:Newman}, which was given a few years later by Bollob{\'a}s and Meir~\cite[Theorem~3]{bollobasTravellingSalesmanProblem1992}.

\begin{figure}[!ht]
\begin{center}
\input{Figures/exNewman}
\end{center}
\caption{Point sets attaining the upper bound in Theorem~\ref{thm:Newman}.}
\label{fig:exNewman}
\end{figure}
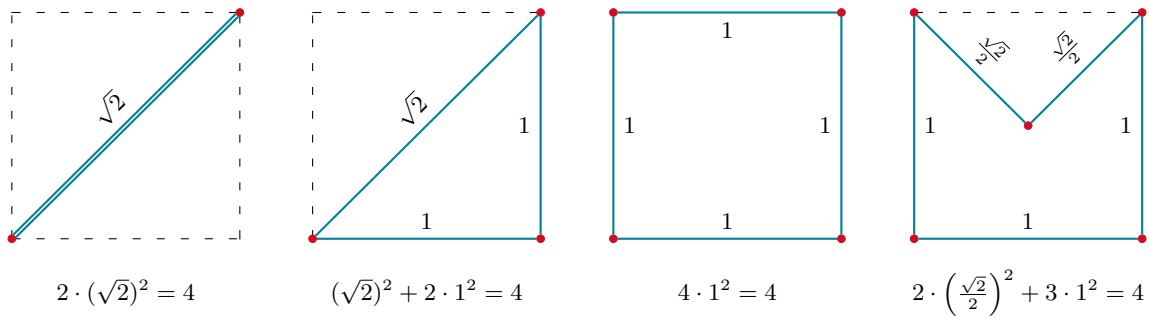

\begin{theorem}[Bollob{\'a}s and Meir]\label{thm:BM}
For any finite set $X \subseteq [0, 1]^k$ of points in the $k$-dimensional unit cube, there exists a Hamiltonian cycle $H$ on $X$ with
\[
S_k(H) = \sum_{e \in H} |e|^k \leq \frac{2}{3} \cdot 9^k \cdot k^{k / 2}.
\]
\end{theorem}

Similar to Theorem~\ref{thm:Newman}, the upper bound in Theorem~\ref{thm:BM} does not depend on the number of points in $X$, only on the dimension $k$.
For a finite set $X \subseteq \R^k$, let $\SHC(X) := \min_H S_k(H)$, where the minimum is over all Hamiltonian cycles $H$ on $X$, and let
\[
\SHC := \sup\left\{ \SHC(X) \,:\, X \subseteq [0, 1]^k,\ X \text{ is a finite set} \right\}.
\]
Let $\SHP(X)$ and $\SHP$, $\SST(X)$ and $\SST$, $\SPM(X)$ and $\SPM$ be defined similarly for, respectively, \textit{Hamiltonian paths}, \textit{spanning trees} and \textit{perfect matchings} instead of Hamiltonian cycles, with only even-sized sets $X$ allowed in the definitions of $\SPM(X)$ and $\SPM$.

\begin{remark}
Related works often used notation like $s_k(E) := (S_k(E))^{\frac{1}{k}}$ and $\sHC(X) := (\SHC(X))^{\frac{1}{k}}$ to avoid dealing with exponents in bounds.
We find using $S_k$ directly more convenient, in particular when discussing the new results of this paper.
To keep notation consistent, we state all results, old and new, in terms of $S_k$.
\end{remark}

\begin{example}\label{ex:lower_bound}
Let $X = \{\fzero, \fone\} \subseteq [0, 1]^k$, where $\fzero = (0, \dots, 0)$ and $\fone = (1, \dots, 1)$.
Then
\[
\SHC \geq \SHC(X) = 2 \cdot k^{k / 2},\quad \SHP \geq \SHP(X) = k^{k / 2},\quad \SST \geq \SST(X) = k^{k / 2},\quad \SPM \geq \SPM(X) = k^{k / 2}.
\]
\end{example}

Combining Theorem~\ref{thm:BM} and Example~\ref{ex:lower_bound}, we see that
\begin{equation}\label{eq:BM-bounds}
2 \cdot k^{k / 2} \leq \SHC \leq \frac{2}{3} \cdot 9^k \cdot k^{k / 2}.    
\end{equation}
In the conclusion of their paper~\cite{bollobasTravellingSalesmanProblem1992}, Bollob{\'a}s and Meir conjectured that $\SHC = 2 \cdot k^{k / 2}$ for all $k$.
The case $k = 1$ is trivial, and due to Theorem~\ref{thm:Newman}, this is also true for $k = 2$.
The exact value of $\SHC$ is not known for any $k > 2$.
Recently, Balogh, Clemen and Dumitrescu~\cite[Theorem~1.3(i)]{baloghTravelingSalesmanProblem2024} disproved the conjecture for $k = 3$ by presenting the following example.

\begin{example}[Balogh, Clemen and Dumitrescu]\label{ex:k=3}
Let $X = \{(0,0,0), (0,1,1), (1,0,1), (1,1,0)\} \subseteq [0, 1]^3$. Then
\[
S^{\mathrm{HC}}_3 \geq S^{\mathrm{HC}}_3(X) = 4 \cdot 2^{3 / 2} > 2 \cdot 3 ^ {3 / 2}.
\]
\end{example}

They proposed an adjusted conjecture~\cite[Conjecture~5.5]{baloghTravelingSalesmanProblem2024}, which remains open for all $k > 2$.

\begin{conjecture}[Bollob{\'a}s--Meir conjecture]\label{conj:BM}
For any finite set of points $X \subseteq [0, 1]^k$, there exists a Hamiltonian cycle $H$ on $X$ with $S_k(H) \leq 2 \cdot k^{k / 2}$ if $k \neq 3$, and $S_3(H) \leq 2^{7/2}$ if $k = 3$.
In other words,
\[
\SHC = \begin{cases}
2 \cdot k^{k / 2} & \text{for } k \neq 3,\\
2^{7/2} & \text{for } k = 3.
\end{cases}
\]
\end{conjecture}

Theorem~\ref{thm:BM} is a direct consequence of two inequalities:
\begin{equation}\label{eq:approx+ball}
\SHC \leq \frac{2}{3} \cdot 3^k \cdot \SST \quad\text{and}\quad \SST \leq 3^k \cdot k^{k / 2}.
\end{equation}

The first inequality follows by approximating a Hamiltonian cycle with a spanning tree, using ideas first developed by Sekanina~\cite{sekaninaOrderingVerticesGraph1963,sekaninaAlgorithmOrderingGraphs1971} in the 60s.
The second inequality is obtained via a ball packing argument: one can show that placing a ball of radius $\frac{|e|}{4}$ on the center of each edge $e$ of the minimal spanning tree yields a collection of disjoint balls, and the inequality follows by bounding their total volume.
Although the upper bound in Theorem~\ref{thm:BM} has since been improved, subsequent works continued to rely on the same combination of ball packing arguments and approximation techniques.

In~\cite[Theorems~3.4 and~5.2]{gordeevChristofidesBasedEUROCOMB2025}, the present author used refined variants of the ball packing argument to obtain the best exact upper bounds to date for $\SST$ and $\SPM$:
\[
\SST \leq 2 \cdot 1.823^k \cdot k^{k / 2} \quad\text{and}\quad \SPM \leq 2 \cdot (\sqrt{2})^k \cdot k^{k / 2}.    
\]
These bounds, together with approximation ideas in the vein of the Christofides algorithm~\cite{christofides1976worst} for the metric travelling salesman problem, led to the best exact upper bound to date for $\SHC$~\cite[Theorem~6.1]{gordeevChristofidesBasedEUROCOMB2025}:
\[
\SHC \leq 1.28 \cdot 5.059^k \cdot k^{k / 2}.
\]

Balogh, Clemen and Dumitrescu~\cite[Theorem~1.5]{baloghTravelingSalesmanProblem2024} established by asymptotic analysis of a certain greedy algorithm (which we will revisit in Section~\ref{sec:asymp}), that for large $k$,
\[
\SST \leq k^{k / 2} \cdot (1 + o_k(1)).
\]
Note that, in view of Example~\ref{ex:lower_bound}, this is asymptotically optimal.
Using approximation, they also obtained an asymptotic upper bound on $\SHC$, which was later improved~\cite[Theorem~6.2]{gordeevChristofidesBasedEUROCOMB2025} to
\[
\SHC \leq 2.65^k \cdot k^{k / 2} \cdot (1 + o_k(1)).
\]

\subsection{New results}
\label{subsec:new}

The main contribution of this work is a new generalization of the ball packing argument.
The key idea is to allow balls to intersect in a controlled manner instead of requiring them to be disjoint.
This not only leads to improved bounds but, more importantly, allows for direct application to Hamiltonian cycles, eliminating the need for approximation.
As a result, we can show that the \nameref{conj:BM} holds asymptotically.

\begin{theorem}\label{thm:main}
For any finite set $X \subseteq [0, 1]^k$, there exists a Hamiltonian cycle $H$ on $X$ such that, if $e_1, e_2 \in H$ are two longest edges of $H$, then
\[
\sum_{e \in H \setminus \{e_1, e_2\}} |e|^k = o_k(k^{k / 2}), \quad\text{so}\quad
S_k(H) = \sum_{e \in H} |e|^k \leq k^{k / 2} \cdot (2 + o_k(1)).
\]
In particular,
\[
\SHC = k^{k / 2} \cdot (2 + o_k(1)).
\]
\end{theorem}

Theorem~\ref{thm:main} is a direct corollary of the following result for Hamiltonian paths, which asserts that Conjecture~5.6 in~\cite{baloghTravelingSalesmanProblem2024}, the analogue of the \nameref{conj:BM} for Hamiltonian paths, also holds asymptotically.

\begin{theorem}\label{thm:HP-asymp}
For any finite set $X \subseteq [0, 1]^k$, there exists a Hamiltonian path $P$ on $X$ such that, if $e_1 \in P$ is the longest edge of $P$, then
\[
\sum_{e \in P \setminus \{e_1\}} |e|^k = o_k(k^{k / 2}), \quad\text{so}\quad
S_k(P) = \sum_{e \in P} |e|^k \leq k^{k / 2} \cdot (1 + o_k(1)).
\]
In particular,
\[
\SHP = k^{k / 2} \cdot (1 + o_k(1)).
\]
\end{theorem}

\begin{proof}[Proof of Theorem~\ref{thm:main} using Theorem~\ref{thm:HP-asymp}]
For a finite set $X \subseteq [0, 1]^k$, let the Hamiltonian path $P$ from Theorem~\ref{thm:HP-asymp} have endpoints $\a, \b \in X$.
Then $H := P \cup \{\a\b\}$ is the desired Hamiltonian cycle.
\end{proof}

Informally speaking, Theorems~\ref{thm:main} and~\ref{thm:HP-asymp} imply that for large $k$ the points of an extremal set $X \subseteq [0, 1]^k$ with $\SHC(X) = k^{k / 2} \cdot (2 + o_k(1))$ (or $\SHP(X) = k^{k / 2} \cdot (1 + o_k(1))$) must be clustered around two opposite vertices of the unit cube $[0, 1]^k$.

Turning to exact bounds, we prove the following theorem, bringing the upper bound on $\SHC$ within a linear factor of the lower bound of Example~\ref{ex:lower_bound}.

\begin{theorem}\label{thm:HC-exact}
For any finite set $X \subseteq [0, 1]^k$, there exists a Hamiltonian cycle $H$ on $X$ such that
\begin{enumerate}[(a)]
\item\label{HC-exact:t} for any integer $t \geq 1$,  $S_k(H) = \sum_{e \in H} |e|^k \leq 2(2t + 1) \cdot \left( 1 + \frac{1}{t} \right)^{k / 2} \cdot k^{k / 2}$;
\item\label{HC-exact:centroid} $S_k(H) \leq 2 \mathrm{e} (k + 1) \cdot k^{k / 2}$, where $\mathrm{e} = 2.718..$ is Euler's number.
\end{enumerate}
In particular,
\[
\SHC \leq 2\mathrm{e}(k + 1) \cdot k^{k / 2}.
\]
\end{theorem}

Any Hamiltonian cycle on a finite set $X$ contains a Hamiltonian path on $X$ as a subset.
Any Hamiltonian path on $X$ is a spanning tree on $X$, and contains a perfect matching on $X$ as a subset when $|X|$ is even.
Thus, Theorems~\ref{thm:HP-asymp} and~\ref{thm:HC-exact} immediately imply the following new upper bounds.

\begin{corollary}
The following hold, where $\mathrm{e}=2.718..$ is Euler's number:
\begin{enumerate}[(a)]
\item $\SHP, \SST \leq 2\mathrm{e}(k + 1) \cdot k^{k / 2}$;
\item $\SPM = k^{k / 2} \cdot (1 + o_k(1))$.
\end{enumerate}
\end{corollary}

For perfect matchings, we show an exact upper bound stronger than those implied by Theorem~\ref{thm:HC-exact}.

\begin{theorem}\label{thm:PM-exact}
For any set $X \subseteq [0, 1]^k$ of even size, there exists a perfect matching $M$ on $X$ such that $S_k(M) \leq 2(k + 1) \cdot k^{k / 2}$.
In other words,
\[
\SPM \leq 2(k + 1) \cdot k^{k / 2}.
\]
\end{theorem}

\subsection{Overview}
\label{subsec:overview}

The rest of the paper is organized as follows.
In Section~\ref{sec:packing}, we present a new generalized ball packing lemma.
We begin by applying it to the setting of perfect matchings, proving Theorem~\ref{thm:PM-exact} in Section~\ref{sec:matching} using classical results of Rankin~\cite{rankinClosestPackingSpherical1955} on the maximal size of spherical codes with negative dot products.
In Section~\ref{sec:HC}, we turn to Hamiltonian cycles, proving Theorem~\ref{thm:HC-exact} via an alternative approach based on properties of centroids of finite point sets.
In Section~\ref{sec:asymp}, we adapt the asymptotic analysis of~\cite[Section~4]{baloghTravelingSalesmanProblem2024} to establish Theorem~\ref{thm:HP-asymp}.
Finally, Section~\ref{sec:concl} concludes with a discussion of related questions.

\section{Half-ball \texorpdfstring{$t$}{t}-fold packing lemma}
\label{sec:packing}

Denote the open ball of radius $r \geq 0$ centered at $\a \in \R^k$ by $\B_r(\a) := \{ \x \in \R^k \,:\, |\x - \a| < r \}$.
For an edge $e = \a\b$ with $\a, \b \in \R^k$, define 
\[
\m_e := \frac{\a + \b}{2},\quad \B_\alpha^e := \B_{\alpha|e|}(\m_e),\quad \hB_\alpha^e := \begin{cases}
\left\{ \x \in \B_\alpha^e \,:\, \langle \x, \m_e\rangle < |\m_e|^2 \right\} &\text{if } \m_e \neq \fzero,\\
\left\{ \x \in \B_\alpha^e \,:\, x_1 < 0 \right\} &\text{if } \m_e = \fzero,
\end{cases}
\]
where $\fzero = (0,\dots, 0) \in \R^k$ is the origin.
In other words, $\m_e$ is the midpoint of $e$, $\B_\alpha^e$ is the open ball of radius $\alpha |e|$ centered at $\m_e$, and $\hB_\alpha^e$ is the open half-ball with the same center and radius oriented towards the origin, or consisting of points with negative first coordinate when $\m_e$ itself is the origin.

The next definition follows the Handbook of Discrete and Computational Geometry~\cite[Section~2.3]{toth20172}.

\begin{definition}
Let $t$ be a positive integer.
A family $\mathcal{A} = \{A_i\}_{i \in I}$ of sets $A_i \subseteq \R^k$, each of finite volume, is a \textit{$t$-fold packing} if each point of $\R^k$ belongs to the interior of at most $t$ sets in $\mathcal{A}$.
A \textit{packing} is a 1-fold packing.
\end{definition}

Let $X \subseteq \R^k$ be a finite set.
For a set $E$ of edges on $X$, let $\mathfrak{B}_\alpha^E := \{ \B_\alpha^e \}_{e \in E}$ and $\mathfrak{hB}_\alpha^E := \{ \hB_\alpha^e \}_{e \in E}$.
The next two lemmas provide an upper bound on $S_k(E)$ when it is known that $\mathfrak{hB}_\alpha^E$ is a $t$-fold packing for some $\alpha$ and $t$.

\begin{lemma}\label{lm:hb-cover}
Let $0 \leq \alpha \leq \frac{1}{2}$, $e = \a\b$ be an edge with $\a, \b \in \R^k$ and $|\a|, |\b| \leq r$.
Then $\hB_\alpha^e \subseteq \B_r(\fzero)$.
\end{lemma}
\begin{proof}
Let $\x \in \hB_\alpha^e$.
If $\m_e = \fzero$, then $|\x| < \alpha |e| \leq \frac{|\a - \b|}{2} \leq \frac{|\a| + |\b|}{2} \leq r$.
Otherwise, $\x = \m_e + \z$ for some $\z \in \R^k$ with $|\z| < \alpha |e|$ and $\langle \m_e + \z, \m_e \rangle < \langle \m_e, \m_e \rangle$, i.e. $\langle \z, \m_e \rangle < 0$.
It follows that
\[
|\x|^2 = |\m_e + \z|^2 = |\m_e|^2 + |\z|^2 + 2\langle \z, \m_e \rangle < |\m_e|^2 + |\z|^2,
\]
so
\[
|\x| < \sqrt{|\m_e|^2 + \alpha^2|e|^2} \leq \sqrt{ \left| \frac{\a + \b}{2} \right|^2 + \left| \frac{\a - \b}{2} \right|^2} = \sqrt{\frac{|\a|^2 + |\b|^2}{2}} \leq r. \qedhere
\]
\end{proof}

\begin{lemma}\label{lm:hb-pack}
Let $X \subseteq \left[ -\frac{1}{2}, \frac{1}{2} \right]^k$ be a finite set and $E$ be a set of edges on $X$.
If $\mathfrak{hB}_\alpha^E$ is a $t$-fold packing for some $0 \leq \alpha \leq \frac{1}{2}$ and a positive integer $t$, then
\begin{equation}\label{eq:hb-pack}
S_k(E) \leq \frac{2t}{(2\alpha)^k} \cdot k^{k / 2}.    
\end{equation}
\end{lemma}
\begin{proof}
Note that $|\a| \leq \frac{\sqrt{k}}{2}$ for all $\a \in X$, thus, by Lemma~\ref{lm:hb-cover}, $\hB_\alpha^e \subseteq \B_{\sqrt{k} / 2}(\fzero)$ for each edge $e \in E$.
Since $\mathfrak{hB}_\alpha^E$ is a $t$-fold packing, the sum of volumes of half-balls $\hB_\alpha^e$ is at most $t$ times the volume of $\B_{\sqrt{k} / 2}(\fzero)$.
Denoting by $V_k$ the volume of the $k$-dimensional unit ball, we have
\[
\sum_{e \in E}\frac{V_k}{2}\alpha^k|e|^k \leq t \cdot V_k \left( \frac{\sqrt{k}}{2} \right)^k, \quad\text{so}\quad S_k(E) \leq \frac{2t}{(2\alpha)^k} \cdot k^{k / 2}. \qedhere
\]
\end{proof}

\begin{remark}
Note that if $\mathfrak{B}_\alpha^E$ is a $t$-fold packing, then so is $\mathfrak{hB}_\alpha^E$, since, by definition, $\hB_\alpha^e \subseteq \B_\alpha^e$ for any edge $e$.
In all applications of Lemma~\ref{lm:hb-pack} that follow, the stronger condition on $\mathfrak{B}_\alpha^E$ is satisfied.
\end{remark}

The idea of using half-balls was introduced by the present author in~\cite{gordeevChristofidesBasedEUROCOMB2025}.
Earlier works applied the volume bound directly to the collection of balls $\mathfrak{B}_\alpha^E$, which produced exponentially weaker bounds on $S_k(E)$: for example, it was shown in~\cite[Lemma~2.4]{baloghTravelingSalesmanProblem2024} that when $E$ is a set of edges on $X \subseteq \left[ -\frac{1}{2}, \frac{1}{2} \right]^k$, the balls of $\mathfrak{B}_{1/4}^E$ in general cannot be covered by a ball of radius smaller than $\frac{\sqrt{5k}}{4}$, which in the case $t = 1$ and $\alpha = 1/4$ results in inequality $S_k(E) \leq (\sqrt{5})^k \cdot k^{k / 2}$.
In the same case, Lemma~\ref{lm:hb-pack} gives $S_k(E) \leq 2^{k + 1} \cdot k^{k / 2}$.

The idea of using $t$-fold packings with $t > 1$ in this context is new.
We will see in the next section that increasing $t$ often allows to obtain $t$-fold packings $\mathfrak{B}_\alpha^E$ with comparatively larger $\alpha$ on the same edge set $E$, leading to exponential improvements in the upper bound~\eqref{eq:hb-pack} on $S_k(E)$.
In Section~\ref{sec:HC}, we will show that $\mathfrak{B}_\alpha^H$ is a $t$-fold packing for a careful choice of $t > 1$, $\alpha$ and a Hamiltonian cycle $H$ on a given point set $X$.
This allows to apply Lemma~\ref{lm:hb-pack} in the case of Hamiltonian cycles directly, without first approximating a Hamiltonian cycle with a less complex edge structure.

\section{Perfect matchings}
\label{sec:matching}

The goal of this section is to prove Theorem~\ref{thm:PM-exact}.
We already have the first ingredient, Lemma~\ref{lm:hb-pack}.
To complete the proof, for any point set $X \subseteq [0, 1]^k$ of even size, it would suffice to find a perfect matching $M$ on $X$ such that $\mathfrak{B}_\alpha^M$ is a $t$-fold packing for some appropriate $t$ and $\alpha$.

Restricting ourselves for a moment to the case $t = 1$, note that $\mathfrak{B}_\alpha^M$ is a packing if and only if for any $e = \a\b, f = \c\d \in M$, the inequality $|\m_e - \m_f| \geq \alpha(|e| + |f|)$ holds.
We can express the distance $|\m_e - \m_f|$ in terms of distances between points $\a, \b, \c, \d$ using the following classical formula for the length of a bimedian of a tetrahedron (see, for example, the textbook of Altshiller-Court~\cite[page 56]{altshiller-courtModernPureSolid1935}).

\begin{lemma}\label{lm:bimedian}
For any $\a, \b, \c, \d \in \R^k$,
\[
\left| \frac{\a + \b}{2} - \frac{\c + \d}{2} \right|^2 = \frac{|\a - \c|^2 + |\b - \d|^2 + |\a - \d|^2 + |\b - \c|^2 - |\a - \b|^2 - |\c - \d|^2}{4}.
\]
\end{lemma}
\begin{proof}
Both sides are equal to 
\[
\frac{|\a|^2 + |\b|^2 + |\c|^2 + |\d|^2}{4} + \frac{\langle \a, \b \rangle + \langle \c, \d \rangle - \langle \a, \c \rangle - \langle \b, \d \rangle - \langle \a, \d \rangle - \langle \b, \c \rangle}{2}.\qedhere
\]
\end{proof}

The following key lemma for perfect matchings easily follows from this formula.

\begin{lemma}\label{lm:match-key}
Let $X \subseteq \R^k$ be a set of even size, and let $M$ be a perfect matching on $X$ with minimal possible $S_2(M)$.
Then, for any $e, f \in M$,
\[
|\m_e - \m_f|^2 \geq \frac{|e|^2 + |f|^2}{4}.
\]
\end{lemma}
\begin{proof}
Denote $e = \a\b$, $f = \c\d$, $M_1 := (M \setminus \{e, f\}) \cup \{\a\c, \b\d\}$, $M_2 := (M \setminus \{e, f\}) \cup \{\a\d, \b\c\}$.
Note that $M_1$ and $M_2$ are also perfect matchings on $X$, thus $S_2(M_1) \geq S_2(M)$ and $S_2(M_2) \geq S_2(M)$, or
\[
|\a - \c|^2 + |\b - \d|^2 \geq |e|^2 + |f|^2 \quad\text{and}\quad |\a - \d|^2 + |\b - \c|^2 \geq |e|^2 + |f|^2.
\]
Then, by Lemma~\ref{lm:bimedian},
\[
| \m_e - \m_f |^2 = \frac{|\a - \c|^2 + |\b - \d|^2 + |\a - \d|^2 + |\b - \c|^2 - |e|^2 - |f|^2}{4} \geq \frac{|e|^2 + |f|^2}{4}.\qedhere
\]
\end{proof}

Combining Lemma~\ref{lm:match-key} with the inequality $|e|^2 + |f|^2 \geq \frac{(|e| + |f|)^2}{2}$, we see that the set $\mathfrak{B}_\alpha^M$ is a packing for $\alpha = \frac{1}{\sqrt{8}}$ and any $S_2$-minimal perfect matching $M$ on $X$. 
Together with Lemma~\ref{lm:hb-pack}, this implies $\SPM \leq 2 \cdot (\sqrt{2})^k \cdot k^{k / 2}$, see~\cite[Theorem~3.4]{gordeevChristofidesBasedEUROCOMB2025}.
However, we can get a much stronger bound by considering $t$-fold packings with $t > 1$ instead.
Namely, we will show that, under the same assumptions as in Lemma~\ref{lm:match-key}, $\mathfrak{B}_{1/2}^M$ is a $(k + 1)$-fold packing, using properties of \textit{spherical codes} with negative dot products.
We will also investigate an alternative approach in Section~\ref{sec:HC}.

Let $\mathrm{S}_k$ denote the unit sphere in $k$ dimensions centered at the origin, that is, $\mathrm{S}_k := \{ \x \in \R^k \,:\, |\x| = 1 \}$.
A $(k, n, s)$-\textit{spherical code} of dimension $k$, size $n$ and maximum dot product $s$ is a set of $n$ points $X \subseteq \mathrm{S}_k$ with $\langle \x, \y \rangle \leq s$ for any distinct $\x, \y \in X$.
Rankin~\cite{rankinClosestPackingSpherical1955} determined the maximal possible size $n$ of a $(k, n, s)$-spherical code when $s \leq 0$.
For convenience, we state this result in the following form, which is directly implied by~\cite[Theorem 1]{rankinClosestPackingSpherical1955}.

\begin{lemma}[Rankin]\label{lm:sph-code-bound}
A $(k, n, s)$-spherical code with $s < 0$ has $n \leq k + 1$.
\end{lemma}

\begin{proposition}\label{prop:match-obtuse}
Let $X \subseteq \R^k$ be a set of even size, and let $M$ be a perfect matching on $X$ with minimal possible $S_2(M)$.
Then $\mathfrak{B}_{1/2}^M$ is a $(k+1)$-fold packing.
\end{proposition}
\begin{proof}
Fix any point $\x \in \R^k$ and let it be contained in the ball $\B_{1/2}^{e_i}$ for $n$ distinct edges $e_1, \dots, e_n \in M$, i.e. $|\m_{e_i} - \x|^2 < |e_i|^2 / 4$ for $1 \leq i \leq n$.
Due to Lemma~\ref{lm:match-key}, for any $i \neq j$ we have
\[
|\m_{e_i} - \m_{e_j}|^2 \geq \frac{|e_i|^2 + |e_j|^2}{4} > |\m_{e_i} - \x|^2 + |\m_{e_j} - \x|^2,
\]
therefore
\[
\langle \m_{e_i} - \x, \m_{e_j} - \x \rangle = \frac{|\m_{e_i} - \x|^2 + |\m_{e_j} - \x|^2 - |(\m_{e_i} - \x) - (\m_{e_j} - \x)|^2}{2} < 0.
\]
Then the set $\left\{\frac{\m_{e_i} - \x}{|\m_{e_i} - \x|} \,:\,1 \leq i \leq n\right\}$ is a $(k, n, s)$-spherical code for some $s < 0$, thus $n \leq k + 1$ by Lemma~\ref{lm:sph-code-bound}.
\end{proof}

\begin{proof}[Proof of Theorem~\ref{thm:PM-exact}]
After a parallel translation, we may assume that $X \subseteq \left[ -\frac{1}{2}, \frac{1}{2} \right]^k$.
Let $M$ be a perfect matching on $X$ which minimizes $S_2(M)$.
Due to Proposition~\ref{prop:match-obtuse}, $\mathfrak{B}_{1/2}^M$ is a $(k + 1)$-fold packing, then so is $\mathfrak{hB}_{1/2}^M$.
Applying Lemma~\ref{lm:hb-pack} with $E = M$, $t = k + 1$ and $\alpha = \frac{1}{2}$, we get
\[
S_k(M) \leq \frac{2(k + 1)}{\left( 2 \cdot \frac{1}{2} \right)^k} \cdot k^{k / 2} = 2(k + 1) \cdot k^{k / 2}. \qedhere
\]
\end{proof}

\section{Hamiltonian cycles}
\label{sec:HC}

At first glance, proving an analogue of Lemma~\ref{lm:match-key} for Hamiltonian cycles seems impossible, since the midpoints of two edges of an (even optimal) cycle can be arbitrarily close to each other, or even coincide as in Example~\ref{ex:lower_bound}.
However, it turns out that such an analogue exists if one considers four edges instead of two.

\begin{lemma}\label{lm:HC-4-edges}
Let $X \subseteq \R^k$ and let $H$ be a Hamiltonian cycle on $X$ with minimal possible $S_2(H)$.
Then, for any four edges $e_1, e_2, e_3, e_4 \in H$, numbered in the order in which they appear on the cycle,
\[
|\m_{e_1} - \m_{e_3}|^2 + |\m_{e_2} - \m_{e_4}|^2 \geq \frac{|e_1|^2 + |e_2|^2 + |e_3|^2 + |e_4|^2}{4}.
\]
\end{lemma}
\begin{proof}
Let $F := \{e_1, e_2, e_3, e_4\}$, $e_i = \a_i\b_i$, $1 \leq i \leq 4$, such that $H = F \cup P_{\b_1\a_2} \cup P_{\b_2\a_3} \cup P_{\b_3\a_4} \cup P_{\b_4\a_1}$, where $P_{\x\y}$ is a (possibly empty) path with endpoints $\x$ and $\y$, see Figure~\ref{fig:altHC} (left).
Denote $F_1 := \{\a_1\a_3, \b_1\b_3, \a_2\b_4, \b_2\a_4\}$, $F_2 := \{ \a_1\b_3, \b_1\a_3, \a_2\a_4, \b_2\b_4\}$, $H_1 := (H \setminus F) \cup F_1$ and $H_2 := (H \setminus F) \cup F_2$.
Note that $H_1$ and $H_2$ are also Hamiltonian cycles on $X$, see Figure~\ref{fig:altHC} (center and right).
It follows that $S_2(H_1) \geq S_2(H)$ and $S_2(H_2) \geq S_2(H)$, thus $S_2(F_1) \geq S_2(F)$ and $S_2(F_2) \geq S_2(F)$.

\begin{figure}[!ht]
\begin{center}
\input{Figures/altHC}
\end{center}
\caption{Hamiltonian cycles differing only in edges with both endpoints in $\{\a_1,\b_1,\a_2,\b_2,\a_3,\b_3,\a_4,\b_4\}$.}
\label{fig:altHC}
\end{figure}
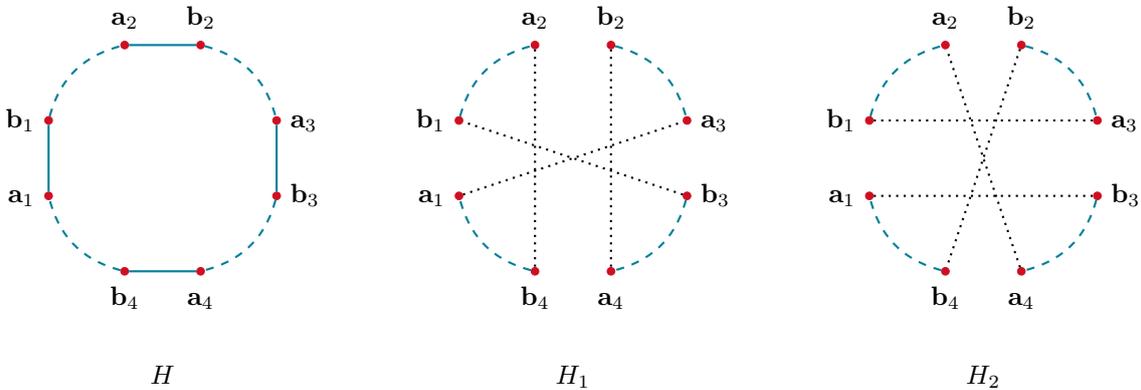

Applying Lemma~\ref{lm:bimedian} two times, for points $\a_1, \b_1, \a_3, \b_3$, and for points $\a_2, \b_2, \a_4, \b_4$, we get
\[
|\m_{e_1} - \m_{e_3}|^2 + |\m_{e_2} - \m_{e_4}|^2 = \frac{S_2(F_1) + S_2(F_2) - S_2(F)}{4} \geq \frac{S_2(F)}{4} = \frac{|e_1|^2 + |e_2|^2 + |e_3|^2 + |e_4|^2}{4}.\qedhere
\]
\end{proof}

As a consequence of Lemma~\ref{lm:HC-4-edges}, under the same notation, either $|\m_{e_1} - \m_{e_3}|^2 \geq \frac{|e_1|^2 + |e_3|^2}{4}$, or $|\m_{e_2} - \m_{e_4}|^2 \geq \frac{|e_2|^2 + |e_4|^2}{4}$, i.e. either $|\m_{e_1} - \m_{e_3}| \geq \frac{|e_1| + |e_3|}{\sqrt{8}}$, or $|\m_{e_2} - \m_{e_4}| \geq \frac{|e_2| + |e_4|}{\sqrt{8}}$.
It follows that $\mathfrak{B}_\alpha^H$ is a 3-fold packing for $\alpha = \frac{1}{\sqrt{8}}$ and any $S_2$-minimal Hamiltonian cycle $H$ on $X$.
Combining this with Lemma~\ref{lm:hb-pack}, we already get $\SHC \leq 6 \cdot 2^{k / 2} \cdot k^{k / 2}$, which is exponentially stronger than all previously known upper bounds on $\SHC$, including the asymptotic ones.

As in previous section, by considering $t$-fold packings with larger $t$, we can push this even further.
One way to do this is to apply Lemma~\ref{lm:sph-code-bound} again.
However, we prove Theorem~\ref{thm:HC-exact} by an alternative and, in a sense, more elementary approach.
Stated in the form of the following lemma, it uses properties of \textit{centroids} of finite point sets.
For the sake of completeness, in Appendix~\ref{ap:alt} we give alternative proofs of exact upper bounds on $\SHC$ using Lemma~\ref{lm:sph-code-bound}, and on $\SPM$ using Lemma~\ref{lm:n-points-dist}.
Both turn out to be somewhat weaker than Theorems~\ref{thm:HC-exact} and~\ref{thm:PM-exact}.

\begin{lemma}\label{lm:n-points-dist}
Let $n \geq 1$ and $\a_1, \dots, \a_n \subseteq \R^k$.
For any $\x \in \R^k$,
\[
\sum_{i = 1}^n |\a_i - \x|^2 \geq \frac{1}{n}\sum_{i < j} |\a_i - \a_j|^2,
\]
with equality only for $\x = \m$, where $\m := \frac{1}{n}\sum_{i = 1}^n \a_i$ is the centroid of points $\a_1, \dots, \a_n$.
\end{lemma}
\begin{proof}
Note that
\[
|\m|^2 = \langle \m, \m \rangle = \sum_{i = 1}^n \frac{|\a_i|^2}{n^2} + \sum_{i < j} \frac{2\langle \a_i, \a_j \rangle}{n^2}.
\]
Then
\begin{align*}
\sum_{i = 1}^n |\a_i - \x|^2 - \frac{1}{n}\sum_{i < j} |\a_i - \a_j|^2 &= n|\x|^2 - 2\sum_{i = 1}^n \langle \a_i, \x \rangle + \sum_{i = 1}^n |\a_i|^2 - \frac{n - 1}{n}\sum_{i = 1}^n |\a_i|^2 + \sum_{i < j} \frac{2\langle \a_i, \a_j \rangle}{n}\\
&= n|\x|^2 - 2n\langle \m, \x \rangle + n|\m|^2 = n|\m - \x|^2 \geq 0,
\end{align*}
with equality only when $\x = \m$.
\end{proof}

\begin{proposition}\label{prop:tsp-centroid}
Let $X \subseteq \R^k$ be a finite set, $H$ be a Hamiltonian cycle on $X$ with minimal possible $S_2(H)$, $t \geq 1$ be a positive integer and $\alpha := \sqrt{\frac{t}{4(t + 1)}}$.
Then $\mathfrak{B}_\alpha^H$ is a $(2t + 1)$-fold packing.
\end{proposition}
\begin{proof}
Fix any point $\x \in \R^k$ and let it be contained in $2n$ balls $\B_{\alpha}^{e_1}, \B_\alpha^{f_1}, \dots, \B_{\alpha}^{e_n}, \B_\alpha^{f_n}$, with the edges appearing on the cycle in the order $e_1, f_1, \dots, e_n, f_n \in H$.
We have
\[
\sum_{i = 1}^n \left( |\m_{e_i} - \x|^2 + |\m_{f_i} - \x|^2 \right) < \sum_{i = 1}^n\alpha^2 \left( |e_i|^2 + |f_i|^2 \right) = \frac{t}{4(t + 1)} \sum_{i = 1}^n \left( |e_i|^2 + |f_i|^2 \right).
\]
On the other hand, applying Lemma~\ref{lm:n-points-dist} two times for the sets of points $\{\m_{e_i}\}_{i = 1}^n$ and $\{\m_{f_i}\}_{i = 1}^n$, and then Lemma~\ref{lm:HC-4-edges} for edges $e_i, f_i, e_j, f_j$ for all $i < j$, we get
\begin{align*}
\sum_{i = 1}^n \left( |\m_{e_i} - \x|^2 + |\m_{f_i} - \x|^2 \right) &\geq \sum_{i < j} \frac{|\m_{e_i} - \m_{e_j}|^2 + |\m_{f_i} - \m_{f_j}|^2}{n}\\
&\geq \sum_{i < j} \frac{|e_i|^2 + |e_j|^2 + |f_i|^2 + |f_j|^2}{4n} = \frac{n - 1}{4n} \sum_{i = 1}^n \left( |e_i|^2 + |f_i|^2 \right).    
\end{align*}
It follows that $\frac{1}{4} - \frac{1}{4n} = \frac{n - 1}{4n} < \frac{t}{4(t + 1)} = \frac{1}{4} - \frac{1}{4(t + 1)}$, which implies $n < t + 1$, i.e. no point is contained in $2(t + 1)$ balls of $\mathfrak{B}_\alpha^H$.
\end{proof}

\begin{proof}[Proof of Theorem~\ref{thm:HC-exact}\ref{HC-exact:t}]
After a parallel translation, we may assume that $X \subseteq \left[ -\frac{1}{2}, \frac{1}{2} \right]^k$.
Let $H$ be a Hamiltonian cycle on $X$ which minimizes $S_2(H)$ and let $\alpha = \sqrt{\frac{t}{4(t + 1)}}$.
By Proposition~\ref{prop:tsp-centroid}, $\mathfrak{B}_\alpha^H$ is a $(2t + 1)$-fold packing, then so is $\mathfrak{hB}_\alpha^H$.
Applying Lemma~\ref{lm:hb-pack} with $E = H$, we get
\[
S_k(H) \leq \frac{2 (2t + 1)}{\left( 2 \cdot \sqrt{\frac{t}{4(t + 1)}} \right)^k} \cdot k^{k / 2} = 2 (2t + 1) \cdot \left( 1 + \frac{1}{t} \right)^{k / 2} \cdot k^{k / 2}. \qedhere
\]
\end{proof}

\begin{lemma}\label{lm:tk-bound}
For any integer $k \geq 1$, there is an integer $t \geq 1$ such that $(2t + 1) \cdot \left( 1 + \frac{1}{t} \right)^{k / 2} \leq (k + 1)\mathrm{e}$,
where $\mathrm{e} = 2.718..$ is Euler's number.
\end{lemma}
\begin{proof}
If $k$ is even, using the inequality $\left( 1 + \frac{1}{n} \right)^n \leq \mathrm{e}$, for $t = k / 2$ we have $(2t + 1) \cdot \left( 1 + \frac{1}{t} \right)^t \leq (k + 1) \mathrm{e}$.
If $k = 1$, for $t = 1$ we have $3 \sqrt{2} \leq 2 \mathrm{e}$.
If $k \geq 3$ is odd, for $t$ such that $k = 2t - 1$, using the inequality $\left( 1 + \frac{1}{n} \right)^n \leq \mathrm{e} \cdot \left( 1 - \frac{1}{2n} + \frac{1}{2n^2} \right)$,
\begin{align*}
(2t + 1) \cdot \left( 1 + \frac{1}{t} \right)^{t - \frac{1}{2}} &\leq (2t + 1) \cdot \frac{k + 1}{2t} \cdot \sqrt{\frac{t}{t + 1}} \cdot \mathrm{e} \cdot \left( 1 - \frac{1}{2t} + \frac{1}{2t^2} \right)\\
&= (k + 1)\mathrm{e} \cdot \frac{(2t + 1) \cdot (2t^2 - t + 1)}{4t^2 \sqrt{t(t + 1)}}.
\end{align*}
Since $k \geq 3$, we have $t \geq 2$ and thus
\[
\frac{(2t + 1) (2t^2 - t + 1)}{4t^2 \sqrt{t(t + 1)}} = \frac{4t^3 + t + 1}{4t^2 \sqrt{t(t + 1)}} \leq \frac{4t^3 + t^2}{4t^2 \sqrt{t(t + 1)}} = \frac{4t + 1}{4\sqrt{t(t + 1)}} \leq 1,
\]
as $t + 1 \leq t^2$ for $t \geq \frac{\sqrt{5} + 1}{2}$, and $(4t + 1)^2 = 16t^2 + 8t + 1 \leq 16t^2 + 16t = \left( 4\sqrt{t(t + 1)} \right)^2$ for $t \geq \frac{1}{8}$.
\end{proof}

\begin{proof}[Proof of Theorem~\ref{thm:HC-exact}\ref{HC-exact:centroid}]
Apply Theorem~\ref{thm:HC-exact}\ref{HC-exact:t} with $t$ given by Lemma~\ref{lm:tk-bound}.
\end{proof}

\section{Asymptotically optimal upper bound}
\label{sec:asymp}

The proof of Theorem~\ref{thm:HP-asymp} below is structured similarly to proofs of Theorems~1.5 and~1.6 in~\cite{baloghTravelingSalesmanProblem2024}.
The following two lemmas correspond to the case $m = 3$ in Lemma~3.3 and to Lemma~4.3 in~\cite{baloghTravelingSalesmanProblem2024}.

\begin{lemma}[Balogh, Clemen and Dumitrescu]\label{lm:3-points}
For any $X \subseteq [0, 1]^k$ with $|X| \geq 3$, there are two distinct points $\p, \q \in X$ with $|\p - \q| \leq \sqrt{\frac{2k}{3}}$.
\end{lemma}

\begin{lemma}[Balogh, Clemen and Dumitrescu]\label{lm:0.999}
There exists $k_0$ such that for all integers $k \geq k_0$ the following holds.
Let $0 < \alpha < 0.99$ and let $Y \subseteq [0, 1]^k$ such that $|\u - \vv| > \alpha \sqrt{k}$ for every two distinct points $\u, \vv \in Y$.
Then $|Y| \cdot \alpha^k \leq 0.999^k$.
\end{lemma}

For a finite set $X \subseteq \R^k$, we will call a set $E$ of edges on $X$ a \textit{path family of size $m$} if $|E| = m$ and $E$ is a union of one or several vertex-disjoint paths.

\begin{remark}\label{rm:PFinHC}
Hamiltonian paths on $X$ are precisely path families on $X$ of size $|X| - 1$.
More generally, for a Hamiltonian cycle $H$ on $X$, any subset $E \subsetneq H$ with $|E| = m < |X|$ is a path family on $X$ of size $m$.
\end{remark}

Arguments used in Section~\ref{sec:HC} for Hamiltonian cycles can be applied with minimal changes to path families of a fixed size.
In particular, we could prove an analogue of Proposition~\ref{prop:tsp-centroid} for path families of a fixed size.
For our purposes, however, the following would suffice.

\begin{proposition}\label{prop:PF-4-edges}
Let $X \subseteq \R^k$ and let $E$ be a path family of size $m$ on $X$ with minimal $S_2(E)$ among all path families of size $m$ on $X$.
Then $\mathfrak{B}_{1 / \sqrt{8}}^E$ is a 3-fold packing.
\end{proposition}
\begin{proof}
Let $H$ be any Hamiltonian cycle on $X$ such that $E \subseteq H$.
Let $e_1, e_2, e_3, e_4 \in E$ be any four distinct edges of $E$, numbered in the order in which they appear on the cycle $H$.
It is sufficient to show that the balls $\B_{1 / \sqrt{8}}^{e_i}$ and $\B_{1 / \sqrt{8}}^{e_j}$ are disjoint for some $i \neq j$.

Let $F := \{e_1, e_2, e_3, e_4\}$, $e_i = \a_i\b_i$, $1 \leq i \leq 4$, such that $H = F \cup P_{\b_1\a_2} \cup P_{\b_2\a_3} \cup P_{\b_3\a_4} \cup P_{\b_4\a_1}$, where $P_{\x\y}$ is a (possibly empty) path with endpoints $\x$ and $\y$.
Denote $F_1 := \{\a_1\a_3, \b_1\b_3, \a_2\b_4, \b_2\a_4\}$, $F_2 := \{ \a_1\b_3, \b_1\a_3, \a_2\a_4, \b_2\b_4\}$, $H_1 := (H \setminus F) \cup F_1$ and $H_2 := (H \setminus F) \cup F_2$.
Since $H_1$ and $H_2$ are also Hamiltonian cycles on $X$ (see Figure~\ref{fig:altHC}), by Remark~\ref{rm:PFinHC}, it follows that $E_1 := (E \setminus F) \cup F_1 \subseteq H_1$ and $E_2 := (E \setminus F) \cup F_2 \subseteq H_2$ are path families on $X$ of size $m$.
Then $S_2(E_1) \geq S_2(E)$ and $S_2(E_2) \geq S_2(E)$, so $S_2(F_1) \geq S_2(F)$ and $S_2(F_2) \geq S_2(F)$.
Applying Lemma~\ref{lm:bimedian} two times, for points $\a_1, \b_1, \a_3, \b_3$, and for points $\a_2, \b_2, \a_4, \b_4$, we get
\[
|\m_{e_1} - \m_{e_3}|^2 + |\m_{e_2} - \m_{e_4}|^2 = \frac{S_2(F_1) + S_2(F_2) - S_2(F)}{4} \geq \frac{S_2(F)}{4} = \frac{|e_1|^2 + |e_2|^2 + |e_3|^2 + |e_4|^2}{4}.
\]
Then, either for $i = 1$ and $j = 3$, or for $i = 2$ and $j = 4$, we have $|\m_{e_i} - \m_{e_j}|^2 \geq \frac{|e_i|^2 + |e_j|^2}{4} \geq \frac{(|e_i| + |e_j|)^2}{8}$, so $|\m_{e_i} - \m_{e_j}| \geq \frac{|e_i| + |e_j|}{\sqrt{8}}$.
As a consequence, the balls $\B_{1 / \sqrt{8}}^{e_i}$ and $\B_{1 / \sqrt{8}}^{e_j}$ are disjoint.
\end{proof}

\begin{proof}[Proof of Theorem~\ref{thm:HP-asymp}]
Let $k$ be sufficiently large.
For $0 \leq i < n$, let $E_i$ be a path family of size $i$ on $X$ with minimal $S_2(E_i)$ among all path families of size $i$ on $X$, and let $i_0$ be minimal $i$ such that $S_2(E_{i + 1}) - S_2(E_i) > \frac{1}{\sqrt{k}}$, or $i_0 = n - 1$ if there is no such $i$.

On the one hand, for any $e \in E_{i_0}$, the set $E_{i_0} \setminus \{e\}$ is a path family of size $i_0 - 1$ on $X$, thus
\[
S_2(E_{i_0}) - |e|^2 = S_2(E_{i_0} \setminus \{e\}) \geq S_2(E_{i_0 - 1}),
\]
i.e., by the definition of $i_0$,
\begin{equation}\label{eq:Ei0-edges-in}
|e|^2 \leq S_2(E_{i_0}) - S_2(E_{i_0 - 1}) \leq \frac{1}{\sqrt{k}} \quad\text{for any } e \in E_{i_0}.    
\end{equation}
On the other hand, for any two endpoints $\u, \vv \in X$ of different paths in $E_{i_0}$, the set $E_{i_0} \cup \{\u\vv\}$ is a path family of size $i_0 + 1$ on $X$, thus 
\[
S_2(E_{i_0}) + |\u\vv|^2 = S_2(E_{i_0} \cup \{\u\vv\}) \geq S_2(E_{i_0 + 1}),
\]
so, using the definition of $i_0$ again,
\begin{equation}\label{eq:Ei0-edges-out}
|\u\vv|^2 \geq S_2(E_{i_0 + 1}) - S_2(E_{i_0}) > \frac{1}{\sqrt{k}} \quad\text{ for any two endpoints } \u, \vv \in X \text{ of different paths in } E_{i_0}.
\end{equation}

Set 
\[
\ell := \left\lceil \log_{1 + \frac{1}{k}} \left( 0.9 k^\frac{3}{4} \right) \right\rceil = O(k \log k),\quad a_i := \frac{\left( 1 + \frac{1}{k} \right)^i}{k^\frac{3}{4}}
\]
for integers $0 \leq i \leq \ell$.
Note that $a_i < 0.9$ for $i < \ell$, and $a_\ell \geq 0.9$.

Construct a Hamiltonian path $P$ by starting from $E_{i_0}$ and successively adding the shortest edge connecting endpoints of two different paths.
For $0 \leq i \leq \ell$, let $P_i \subseteq P$ be the path family consisting of all edges $e \in P$ with $|e| \leq a_i \sqrt{k}$.
Then $P_0 \subseteq P_1 \subseteq \dots \subseteq P_\ell \subseteq P$, and for any two endpoints $\u, \vv \in X$ of different paths in $P_i$, we have $|\u - \vv| > a_i\sqrt{k}$.

Due to~\eqref{eq:Ei0-edges-in} and~\eqref{eq:Ei0-edges-out}, we have $P_0 = E_{i_0}$, so $P_0$ is $S_2$-minimal among all path families of size $i_0$.
Then, by Proposition~\ref{prop:PF-4-edges}, $\mathfrak{B}_{1/\sqrt{8}}^{P_0}$ is a 3-fold packing.
It is well-known that the volume $V_k$ of the $k$-dimensional unit ball is
\[
V_k = \begin{cases}
\frac{\pi^{k / 2}}{(k / 2)!} &\text{if } k \text{ is even},\\
\frac{2^k \cdot \pi^{(k - 1) / 2}((k - 1) / 2)!}{k!} &\text{if } k \text{ is odd}.
\end{cases}
\]
By Stirling's approximation, $V_k \sim \frac{1}{\sqrt{k\pi}}( \frac{2\pi\mathrm{e}}{k} )^{k / 2}$.
Since $|e| \leq a_0\sqrt{k} = k^{-1/4}$ for any $e \in P_0$, we have $\bigcup_{e \in P_0} \B_{1/\sqrt{8}}^e \subseteq \left[-k^{-1/4}, 1 + k^{-1/4} \right]$, so
\[
\sum_{e \in P_0} V_k \cdot \left( \frac{|e|}{\sqrt{8}} \right)^k \leq 3 \cdot \left( 1 + 2k^{-1/4} \right)^k, \quad\text{thus}\quad \sum_{e \in P_0} |e|^k \leq \frac{3 \cdot 8^{k / 2} \cdot \left( 1 + 2k^{-1/4} \right)^k}{V_k} \leq 0.69^k \cdot k^{k / 2},
\]
for sufficiently large $k$.

Now, for $0 \leq i < \ell$, let $X_i$ be a set containing one endpoint of each path in the path family $P_i$, so $|X_i| = |X| - |P_i|$.
Then $|\u - \vv| > a_i\sqrt{k}$ for any distinct $\u, \vv \in X_i$, and $|P_{i + 1} \setminus P_i| \leq |P| - |P_i| = (|X| - 1) - (|X| - |X_i|) < |X_i|$.
By Lemma~\ref{lm:0.999}, we have $|X_i| \cdot a_i^k \leq 0.999^k$.
Then
\[
\sum_{e \in P_{i + 1} \setminus P_i} |e|^k < |X_i| \cdot \left( a_{i + 1}\sqrt{k} \right)^k  = |X_i| \cdot \left( a_i\sqrt{k} \right)^k \cdot \left( 1 + \frac{1}{k} \right)^k \leq \mathrm{e} \cdot 0.999^k \cdot k^{k / 2},
\]
where $\mathrm{e}=2.718..$ is Euler's number.
Thus
\[
\sum_{e \in P_\ell} |e|^k = \sum_{e \in P_0} |e|^k + \sum_{i = 1}^{\ell - 1} \sum_{e \in P_{i + 1} \setminus P_i} |e|^k < 0.69^k \cdot k^{k / 2} + \mathrm{e} \ell \cdot 0.999^k \cdot k^{k / 2} = o_k(k^{k / 2})
\]
for sufficiently large $k$ since $\ell = O(k \log k)$.

Since for any two endpoints $\u, \vv$ of different paths in $P_\ell$ we have $|\u - \vv| > a_\ell\sqrt{k} \geq 0.9\sqrt{k} > \sqrt{\frac{2k}{3}}$,
by Lemma~\ref{lm:3-points}, $P_\ell$ consists of at most two paths, i.e. $|P \setminus P_\ell| \leq 1$, so $P$ is the required Hamiltonian path.
\end{proof}

\section{Concluding remarks}
\label{sec:concl}

Theorems~\ref{thm:main} and~\ref{thm:HC-exact} are non-constructive and do not give a computationally efficient way of finding the Hamiltonian cycle in question.
To obtain a constructive result, we can combine these theorems with an approximation algorithm.

A complete graph on a vertex set $X$ with edge weights $w(a, b)$, $a, b \in X$, satisfies \textit{the $\tau$-relaxed triangle inequality} for $\tau \geq 1$ if $w(a, c) \leq \tau \cdot (w(a, b) + w(b, c))$ for all $a, b, c \in X$.
Bender and Chekuri~\cite[Theorem~1]{benderPerformanceGuaranteesTSP2000} gave a polynomial-time $4\tau$-approximation algorithm for the travelling salesman problem in graphs satisfying the $\tau$-relaxed triangle inequality.
By H{\"o}lder's inequality, for any $\a, \b, \c \in \R^k$ we have
\[
|\a - \c|^k \leq \left( |\a - \b| + |\a - \c| \right)^k \leq \left( 1^\frac{k}{k - 1} + 1^\frac{k}{k - 1} \right)^{k - 1} \left( |\a - \b|^k + |\b - \c|^k \right) = 2^{k - 1} \left( |\a - \b|^k + |\b - \c|^k \right).
\]
In other words, for a finite set $X \subseteq \R^k$, the complete graph $K(X)$ with edge weights $w(\a, \b) := |\a - \b|^k$, $\a, \b \in X$, satisfies the $2^{k - 1}$-relaxed triangle inequality, which implies the following corollary.

\begin{corollary}\label{cor:HC-alg}
For any $X \subseteq [0, 1]^k$, $|X| = n$, one can construct a Hamiltonian cycle $H$ on $X$ in time polynomial in $n$ such that $S_k(H) \leq 2^{k + 2} \cdot \mathrm{e}(k + 1) \cdot k^{k / 2}$, or $S_k(H) \leq 2^{k + 2} \cdot k^{k / 2} \cdot (1 + o_k(1))$.
\end{corollary}

In the conclusion of their paper~\cite{baloghTravelingSalesmanProblem2024}, Balogh, Clemen and Dumitrescu proposed a variant of the problem with a different condition in place of $X \subseteq [0, 1]^k$: they asked to determine the extremal value of $\SHC(X)$ over all finite sets $X \subseteq \R^k$ with $\diam X \leq 1$.
They noted that the set of $k + 1$ vertices $X_k$ of the unit simplex in $\R^k$ satisfies $\diam X_k = 1$ and $\SHC(X_k) = k + 1$, and asked if this is optimal.
With present tools, we can prove the following upper bound.

\begin{theorem}\label{thm:diam}
For any finite set $X \subseteq \R^k$ with $\diam X \leq 1$, there exists a Hamiltonian cycle $H$ on $X$ with 
\[
S_k(H) \leq 2 \mathrm{e} \left( \frac{2k}{k + 1} \right)^{k / 2} \cdot (k + 1) \leq 2 \mathrm{e} \cdot 2^{k / 2} \cdot (k + 1),
\]
where $\mathrm{e} = 2.718..$ is Euler's number.
\end{theorem}
\begin{proof}
According to Jung's theorem~\cite{jung1901ueber} (see~\cite[Theorem~3.3]{gruberConvexDiscreteGeometry2007} for a more contemporary treatment), there exists a closed ball of radius $r := \sqrt{\frac{k}{2(k + 1)}}$ that contains $X$.
Then, after a parallel translation, we may assume that $|\a| \leq r$ for any $\a \in X$.
Let $H$ be a Hamiltonian cycle on $X$ which minimizes $S_2(H)$, and let $\alpha = \sqrt{\frac{t}{4(t + 1)}}$, where $t$ is given by Lemma~\ref{lm:tk-bound}, thus
\[
\frac{2t + 1}{\alpha^k} \leq (k + 1) \mathrm{e} \cdot 2^k.
\]
By Proposition~\ref{prop:tsp-centroid}, $\mathfrak{B}_\alpha^H$ is a $(2t + 1)$-fold packing, then so is $\mathfrak{hB}_\alpha^H$.
Due to Lemma~\ref{lm:hb-cover}, $\hB_\alpha^e \subseteq \B_r(\fzero)$ for each edge $e \in H$.
Denoting by $V_k$ the volume of the $k$-dimensional unit ball, we have
\[
\sum_{e \in H} \frac{V_k}{2}\alpha^k|e|^k \leq (2t + 1) \cdot V_k r^k, \quad\text{so}\quad S_k(H) \leq 2 \cdot \frac{2t + 1}{\alpha^k} \cdot r^k \leq 2\mathrm{e}(k + 1) \cdot \left(  \frac{2k}{k + 1} \right)^{k / 2}.\qedhere
\]
\end{proof}

\section*{Acknowledgements}

Alexey Gordeev is a grantee of the Momentum MSCA Programme, which has been co-funded by the European Commission through the HORIZON-MSCA-2023-COFUND programme and the Secretariat of the Hungarian Academy of Sciences (MTA).

\bibliographystyle{abbrv}
\bibliography{main}

\noindent\textit{Email address:} \texttt{gordeev.aleksei@renyi.hu}\\
Department of Mathematics and Mathematical Statistics, Ume\r{a} University, Universitetstorget 4, 901 87 Ume\r{a}, Sweden;\\
HUN-REN Alfr\'ed R\'enyi Institute of Mathematics, Re\'altanoda utca 13--15., H-1053 Budapest,  Hungary;\\
MTA--HUN-REN RI Lend\"ulet ``Momentum'' Arithmetic Combinatorics Research Group, Re\'altanoda utca 13--15., H-1053 Budapest,  Hungary.

\newpage

\appendix
\counterwithin{figure}{section}
\counterwithin{table}{section}

\clearpage
\section{Alternative exact upper bounds}\label{ap:alt}

In this section, we prove an exact upper bound on $\SHC$ using properties of spherical codes, i.e. Lemma~\ref{lm:sph-code-bound}, and an exact upper bound on $\SPM$ using properties of centroids, i.e. Lemma~\ref{lm:n-points-dist}.
The resulting Theorems~\ref{thm:HC-exact-alt} and~\ref{thm:PM-exact-alt} are slightly weaker than, respectively, Theorems~\ref{thm:HC-exact} and~\ref{thm:PM-exact}.

\begin{proposition}\label{prop:tsp-obtuse}
Let $X \subseteq \R^k$ be a finite set, and let $H$ be a Hamiltonian cycle on $X$ with minimal possible $S_2(H)$.
Then $\mathfrak{B}_{1/2}^H$ is a $3(k + 1)$-fold packing.
\end{proposition}
\begin{proof}
Fix any point $\x \in \R^k$ and let it be contained in the ball $\B_{1/2}^{e_i}$ for $n$ distinct edges appearing in the cycle in the order $e_1, \dots, e_n \in H$.
Denoting $\m_i := \m_{e_i}$, we have $|\m_i - \x|^2 < |e_i|^2 / 4$ for $1 \leq i \leq n$.
By Lemma~\ref{lm:HC-4-edges}, for any $1 \leq i_1 < i_2 < i_3 < i_4 \leq n$ we have
\[
| \m_{i_1} - \m_{i_3} |^2 + | \m_{i_2} - \m_{i_4} |^2 \geq \frac{|e_{i_1}|^2 + |e_{i_2}|^2 + |e_{i_3}|^2 + |e_{i_4}|^2}{4} > |\m_{i_1} - \x|^2 + |\m_{i_2} - \x|^2 + |\m_{i_3} - \x|^2 + |\m_{i_4} - \x|^2.
\]
Thus, either $|\m_{i_1} - \m_{i_3}|^2 > |\m_{i_1} - \x|^2 + |\m_{i_3} - \x|^2$, and so
\[
\langle \m_{i_1} - \x, \m_{i_3} - \x \rangle = \frac{|\m_{i_1} - \x|^2 + |\m_{i_3} - \x|^2 - |(\m_{i_1} - \x) - (\m_{i_3} - \x)|^2}{2} < 0,
\]
or, similarly, $\langle \m_{i_2} - \x, \m_{i_4} - \x \rangle < 0$.

Consider an auxiliary graph $G$ with the vertex set $\{1, \dots, n\}$ and an edge between $i \neq j$ whenever $\langle \m_i - \x, \m_j - \x \rangle \geq 0$.
Draw this graph on the plane, placing vertices at the corners of a regular $n$-sided polygon, and drawing edges of $G$ as segments.
We have shown that for any $i_1 < i_2 < i_3 < i_4$ at most one of $i_1i_3$ and $i_2i_4$ is an edge of $G$, which implies that the drawing has no crossing edges.
Thus $G$ has a planar drawing with all vertices belonging to the outer face, i.e. $G$ is \textit{outerplanar}.
Then (see, for example, Proskurowski and Sys{\l}o~\cite[Theorem 1]{proskurowskiEfficientVertexEdgeColoring1986}) vertices of $G$ can be colored with three colors such that no edge has endpoints of the same color.
Taking the largest color class, we obtain a subset $I \subseteq \{1, \dots, n\}$, $|I| \geq \lceil \frac{n}{3} \rceil$, such that for any $i, j \in I$, $i \neq j$, $ij$ is not an edge of $G$, that is, $\langle \m_i - \x, \m_j - \x \rangle < 0$.
Then the set $\{\frac{\m_i - \x}{|\m_i - \x|} \,:\,i \in I\}$ is a $(k, |I|, s)$ spherical code for some $s < 0$, thus, due to Lemma~\ref{lm:sph-code-bound}, $\lceil \frac{n}{3} \rceil \leq |I| \leq k + 1$, i.e. $n \leq 3(k + 1)$.
\end{proof}

\begin{theorem}\label{thm:HC-exact-alt}
For any finite set $X \subseteq [0, 1]^k$, there exists a Hamiltonian cycle $H$ on $X$ such that $S_k(H) \leq 6(k + 1) \cdot k^{k / 2}$.
In other words,
\[
\SHC \leq 6(k + 1) \cdot k^{k / 2}.
\]
\end{theorem}
\begin{proof}
After a parallel translation, we may assume that $X \subseteq \left[ -\frac{1}{2}, \frac{1}{2} \right]^k$.
Let $H$ be a Hamiltonian cycle on $X$ which minimizes $S_2(H)$.
By Proposition~\ref{prop:tsp-obtuse}, $\mathfrak{B}_{1/2}^H$ is a $3(k + 1)$-fold packing, then so is $\mathfrak{hB}_{1/2}^H$.
Applying Lemma~\ref{lm:hb-pack} with $E = H$ and $\alpha = \frac{1}{2}$, we get
\[
S_k(H) \leq \frac{2 \cdot 3(k + 1)}{\left( 2 \cdot \frac{1}{2} \right)^k} \cdot k^{k / 2} = 6(k + 1) \cdot k^{k / 2}. \qedhere
\]
\end{proof}

\begin{proposition}\label{prop:match-centroid}
Let $X \subseteq \R^k$ be a set of even size, $M$ be a perfect matching on $X$ with minimal possible $S_2(M)$, $t \geq 1$ and $\alpha := \sqrt{\frac{t}{4(t + 1)}}$.
Then $\mathfrak{B}_\alpha^M$ is a $t$-fold packing.
\end{proposition}
\begin{proof}
Fix any point $\x \in \R^k$ and let it be contained in the ball $\B_{\alpha}^{e_i}$ for $n$ distinct edges $e_1, \dots, e_n \in M$.
Then
\[
\sum_{i = 1}^n |\m_{e_i} - \x|^2 < \sum_{i = 1}^n\alpha^2 |e_i|^2 = \frac{t}{4(t + 1)} \sum_{i = 1}^n |e_i|^2.
\]
On the other hand, using Lemma~\ref{lm:n-points-dist} and then Lemma~\ref{lm:match-key}, we have
\[
\sum_{i = 1}^n |\m_{e_i} - \x|^2 \geq \frac{1}{n} \sum_{i < j} |\m_{e_i} - \m_{e_j}|^2 \geq \sum_{i < j} \frac{|e_i|^2 + |e_j|^2}{4n} = \frac{n - 1}{4n} \sum_{i = 1}^n |e_i|^2.
\]
It follows that $\frac{1}{4} - \frac{1}{4n} = \frac{n - 1}{4n} < \frac{t}{4(t + 1)} = \frac{1}{4} - \frac{1}{4(t + 1)}$, which implies $n < t + 1$, i.e. $n \leq t$.
\end{proof}

\begin{theorem}\label{thm:PM-exact-alt}
For any integer $k \geq 2$ and any set $X \subseteq [0, 1]^k$ of even size, there exists a perfect matching $M$ on $X$ such that
\begin{enumerate}[(a)]
\item\label{PM-exact-alt:t} for any integer $t \geq 1$, $S_k(M) \leq 2t \cdot \left( 1 + \frac{1}{t} \right)^{k / 2} \cdot k^{k / 2}$;
\item\label{PM-exact-alt:centroid} $S_k(M) \leq k\mathrm{e} \cdot k^{k / 2}$, where $\mathrm{e}=2.718..$ is Euler's number;
\end{enumerate}
In particular,
\[
\SPM \leq k\mathrm{e} \cdot k^{k / 2}.
\]
\end{theorem}
\begin{proof}\leavevmode
\begin{enumerate}[(a)]
\item After a parallel translation, we may assume that $X \subseteq \left[ -\frac{1}{2}, \frac{1}{2} \right]^k$.
Let $M$ be a perfect matching on $X$ which minimizes $S_2(M)$ and let $\alpha = \sqrt{\frac{t}{4(t + 1)}}$.
Due to Proposition~\ref{prop:match-centroid}, $\mathfrak{B}_\alpha^M$ is a $t$-fold packing, then so is $\mathfrak{hB}_\alpha^M$.
Applying Lemma~\ref{lm:hb-pack} with $E = M$, we get
\[
S_k(M) \leq \frac{2t}{\left( 2 \cdot \sqrt{\frac{t}{4(t + 1)}} \right)^k} \cdot k^{k / 2} = 2t \cdot \left( 1 + \frac{1}{t} \right)^{k / 2} \cdot k^{k / 2}.
\]
\item If $k$ is even, then, using item~\ref{PM-exact-alt:t} with $t = k / 2$ and the inequality $(1 + \frac{1}{n})^n \leq \mathrm{e}$,
\[
S_k(M) \leq  2 \cdot \frac{k}{2} \cdot \left( 1 + \frac{1}{t} \right)^t \cdot k^{k / 2} \leq k\mathrm{e}\cdot k^{k / 2}.
\]
If $k$ is odd, then, using item~\ref{PM-exact-alt:t} with $t$ such that $k = 2t + 1$,
\[
S_k(M) \leq 2t \cdot\left( 1 + \frac{1}{t} \right)^{t + \frac{1}{2}} \cdot k^{k / 2} = 2 \sqrt{t(t + 1)} \left( 1 + \frac{1}{t} \right)^t \cdot k^{k / 2} \leq 2\sqrt{t(t + 1)} \cdot \mathrm{e} \cdot k^{k / 2}.
\]
Combining this with the inequality $2\sqrt{t(t + 1)} \leq 2\frac{t + (t + 1)}{2} = k$, we again have $S_k(M) \leq k\mathrm{e} \cdot k^{k / 2}$.\qedhere
\end{enumerate}
\end{proof}

\end{document}

%% file: Figures/exNewman.tex
\tikzset{cubevert/.style={circle,draw,fill,inner sep=0,minimum size=0.005cm}, vertex/.style={circle,draw=myred,fill=myred,inner sep=0pt,minimum size=0.1cm}, every edge/.style={draw}}
\begin{tabular}{c c c c c c c}
\begin{tikzpicture}
    \node[cubevert] (00) at (0,0) {};
    \node[cubevert] (10) at (3,0) {};
    \node[cubevert] (01) at (0,3) {};
    \node[cubevert] (11) at (3,3) {};
    \draw (00) edge[loosely dashed] (01)
          (10) edge[loosely dashed] (11)
          (00) edge[loosely dashed] (10)
          (01) edge[loosely dashed] (11);

  \node[vertex] (a) at (3,3) {};
  \node[vertex] (b) at (0,0) {};

  \draw (a.200) edge[myblue,thick] node[above,sloped,black]{\small $\sqrt{2}$} (b.70)
        (b.20) edge[myblue,thick] (a.250);
\end{tikzpicture}
& &
\begin{tikzpicture}
    \node[cubevert] (00) at (0,0) {};
    \node[cubevert] (01) at (0,3) {};
    \node[cubevert] (11) at (3,3) {};
    \draw (00) edge[loosely dashed] (01)
          (01) edge[loosely dashed] (11);

  \node[vertex] (a) at (3,3) {};
  \node[vertex] (b) at (3,0) {};
  \node[vertex] (c) at (0,0) {};
  
  \draw (a) edge[myblue,thick] node[left,black]{\small $1$} (b)
        (b) edge[myblue,thick] node[above,black]{\small $1$} (c)
        (c) edge[myblue,thick] node[above,sloped,black]{\small $\sqrt{2}$} (a);
\end{tikzpicture}
& &
\begin{tikzpicture}
  \node[vertex] (a) at (3,3) {};
  \node[vertex] (b) at (3,0) {};
  \node[vertex] (c) at (0,0) {};
  \node[vertex] (d) at (0,3) {};
  
  \draw (a) edge[myblue,thick] node[left,black]{\small $1$} (b)
        (b) edge[myblue,thick] node[above,black]{\small $1$} (c)
        (c) edge[myblue,thick] node[right,black]{\small $1$} (d)
        (d) edge[myblue,thick] node[below,black]{\small $1$} (a);
\end{tikzpicture}
& &
\begin{tikzpicture}
    \node[cubevert] (01) at (0,3) {};
    \node[cubevert] (11) at (3,3) {};
    \draw (01) edge[loosely dashed] (11);

  \node[vertex] (a) at (3,3) {};
  \node[vertex] (b) at (3,0) {};
  \node[vertex] (c) at (0,0) {};
  \node[vertex] (d) at (0,3) {};
  \node[vertex] (e) at (1.5,1.5) {};
  
  \draw (a) edge[myblue,thick] node[left,black]{\small $1$} (b)
        (b) edge[myblue,thick] node[above,black]{\small $1$} (c)
        (c) edge[myblue,thick] node[right,black]{\small $1$} (d)
        (d) edge[myblue,thick] node[above,sloped,black]{\small $\frac{\sqrt{2}}{2}$} (e)
        (e) edge[myblue,thick] node[above,sloped,black]{\small $\frac{\sqrt{2}}{2}$} (a);
\end{tikzpicture}\\
\\[-0.2cm]
\small $2 \cdot (\sqrt{2})^2 = 4$ & & \small $(\sqrt{2})^2 + 2 \cdot 1^2 = 4$ & & \small $4 
\cdot 1^2 = 4$ & & \small $2 \cdot \left( \frac{\sqrt{2}}{2} \right)^2 + 3 \cdot 1^2 = 4$
\end{tabular}

%% file: Figures/altHC.tex
\tikzset{vertex/.style={circle,draw=myred,fill=myred,inner sep=0pt,minimum size=0.1cm}, every edge/.style={draw=myblue,thick}}
\begin{tabular}{ccccc}
\begin{tikzpicture}
  \node[vertex,label=left:{$\phantom{|}\a_1$}] (a) at (0,1) {};
  \node[vertex,label=left:{$\phantom{|}\b_1$}] (b) at (0,2) {};
  \node[vertex,label=above:{$\phantom{|}\a_2\phantom{|}$}] (c) at (1,3) {};
  \node[vertex,label=above:{$\phantom{|}\b_2\phantom{|}$}] (d) at (2,3) {};
  \node[vertex,label=right:{$\a_3\phantom{|}$}] (e) at (3,2) {};
  \node[vertex,label=right:{$\b_3\phantom{|}$}] (f) at (3,1) {};
  \node[vertex,label=below:{$\phantom{|}\a_4\phantom{|}$}] (g) at (2,0) {};
  \node[vertex,label=below:{$\phantom{|}\b_4\phantom{|}$}] (h) at (1,0) {};
  
  \draw (a) edge (b)
        (b) edge[bend left,dashed] (c)
        (c) edge (d)
        (d) edge[bend left,dashed] (e)
        (e) edge (f)
        (f) edge[bend left,dashed] (g)
        (g) edge (h)
        (h) edge[bend left,dashed] (a);
\end{tikzpicture}
& &
\begin{tikzpicture}
  \node[vertex,label=left:{$\phantom{|}\a_1$}] (a) at (0,1) {};
  \node[vertex,label=left:{$\phantom{|}\b_1$}] (b) at (0,2) {};
  \node[vertex,label=above:{$\phantom{|}\a_2\phantom{|}$}] (c) at (1,3) {};
  \node[vertex,label=above:{$\phantom{|}\b_2\phantom{|}$}] (d) at (2,3) {};
  \node[vertex,label=right:{$\a_3\phantom{|}$}] (e) at (3,2) {};
  \node[vertex,label=right:{$\b_3\phantom{|}$}] (f) at (3,1) {};
  \node[vertex,label=below:{$\phantom{|}\a_4\phantom{|}$}] (g) at (2,0) {};
  \node[vertex,label=below:{$\phantom{|}\b_4\phantom{|}$}] (h) at (1,0) {};
  
  \draw (b) edge[bend left,dashed] (c)
        (d) edge[bend left,dashed] (e)
        (f) edge[bend left,dashed] (g)
        (h) edge[bend left,dashed] (a)
        (a) edge[black,dotted] (e)
        (d) edge[black,dotted] (g)
        (f) edge[black,dotted] (b)
        (c) edge[black,dotted] (h);
\end{tikzpicture}
& &
\begin{tikzpicture}
  \node[vertex,label=left:{$\phantom{|}\a_1$}] (a) at (0,1) {};
  \node[vertex,label=left:{$\phantom{|}\b_1$}] (b) at (0,2) {};
  \node[vertex,label=above:{$\phantom{|}\a_2\phantom{|}$}] (c) at (1,3) {};
  \node[vertex,label=above:{$\phantom{|}\b_2\phantom{|}$}] (d) at (2,3) {};
  \node[vertex,label=right:{$\a_3\phantom{|}$}] (e) at (3,2) {};
  \node[vertex,label=right:{$\b_3\phantom{|}$}] (f) at (3,1) {};
  \node[vertex,label=below:{$\phantom{|}\a_4\phantom{|}$}] (g) at (2,0) {};
  \node[vertex,label=below:{$\phantom{|}\b_4\phantom{|}$}] (h) at (1,0) {};
  
  \draw (b) edge[bend left,dashed] (c)
        (d) edge[bend left,dashed] (e)
        (f) edge[bend left,dashed] (g)
        (h) edge[bend left,dashed] (a)
        (a) edge[black,dotted] (f)
        (g) edge[black,dotted] (c)
        (b) edge[black,dotted] (e)
        (d) edge[black,dotted] (h);
\end{tikzpicture}\\
\\
$H$ & & $H_1$ & & $H_2$
\end{tabular}